\documentclass[oneside]{amsart}
\usepackage{local}

\begin{document}

\title[Weyl-Kac formula in Deligne's category]%
{Weyl-Kac character formula for affine Lie algebra in Deligne's category}
\author{Aleksei Pakharev}

\maketitle

\begin{abstract}
  We study the characters of simple modules in the parabolic BGG category of the
  affine Lie algebra in Deligne's category.
  More specifically, we take the limit of Weyl-Kac formula to
  compute the character of the irreducible quotient $L(X,k)$
  of the parabolic Verma module $M(X,k)$ of level $k$, where $X$ is an indecomposable object of
  Deligne's category $\drep(GL_t)$, $\drep(O_t)$, or $\drep(Sp_t)$, under conditions
  that the highest weight of $X$ plus the level gives a fundamental weight,
  $t$ is transcendental, and the base field $\kk$ has characteristic $0$.
  We compare our result to the partial result in \cite[Problem 6.2]{EtingofII},
  and evaluate the characters to the categorical dimensions
  to get a categorical interpretation of the Nekrasov-Okounkov hook length formula, \cite[Formula (6.12)]{NO06}.
\end{abstract}

\tableofcontents

\addtocontents{toc}{\protect\setcounter{tocdepth}{1}}

\section{Introduction}

Let $\kk$ be a field of characteristic $0$.
Deligne's categories $\drep(GL_t)$, $\drep(O_t)$, and $\drep(Sp_t)$ are certain
interpolations of the corresponding representation categories of classical
linear groups, where $t$ is an element of the base field $\kk$.
For further details on $\drep(GL_t)$, see
\cite[Section 10]{DeligneSecond},
\cite[Section 2.5]{EtingofII}, 
\cite[Section 3]{CW2011}.
For the details on $\drep(O_t)$, see
\cite[Section 9]{DeligneSecond},
\cite[Section 2.6]{EtingofII}, 
\cite[Section 2]{CH2015}.
Also see
\cite[Section 9.5]{DeligneSecond},
\cite[Section 2.6]{EtingofII} for an explanation of how $\drep(Sp_t)$ is (almost)
equivalent to $\drep(O_{-t})$.

In Deligne’s category $\drep(G_t)$, one can do the usual representation theory.
For example, in \cite[Section 6]{EtingofII} Etingof defines the affine Lie
algebra $\aff{\lie{g}}_t$ in $\drep(G_t)$.
Given an indecomposable object $X$ in $\drep(G_t)$ and a level $k\in\kk$,
we can consider the parabolic Verma module $M(X,k)$, and
then its irreducible quotient $L(X,k)$.
In \cite[Problem 6.2]{EtingofII} Etingof asks what is the character of $L(X,k)$,
and computes it in the particular case of $k=1$ and $X$ being the identity
object of $\drep(G_t)$ by giving the limit of the character formula derived from
the Frenkel-Kac vertex operator construction.

We provide a description of the character of $L(X,k)$ when $t$ is transcendental,
and the level together with the highest weight of $X$ form a dominant weight,
in particular, when $k$ is a sufficiently big integer. Putting our method into one sentence,
we are calculating the limit of the Weyl-Kac identity (\cref{stm:relative_finite_kac})
in usual representation categories $\rep(G_n)$ as $n$ goes to $t$. To
do that, we need to establish the limit of $L(X,k)$, and the limit of the
alternating sum over the Weyl group.

The paper goes as follows. In \cref{sec:preliminaries} we review the related definitions of Deligne's
category and an affine Lie algebra in a tensor category, some facts about
Weyl groups, and the statement of Weyl-Kac character series. In \cref{sec:m_l_coherence},
we prove that it is valid to take the limit of
the irreducible objects $L(X,k)$ considered over $G_n$ to be equal to $L(X,k)$
considered over $G_t$.

The non-trivial part of the paper is getting the limit of the right hand side
of the Weyl-Kac formula. To do that, in \cref{sec:limit_dynkin} we introduce a Dynkin diagram
for Deligne’s category as the limit of the finite case Dynkin diagrams around the affine vertex.
The crucial piece is \cref{stm:any_wx_is_stable}, where we prove that restricting the grading
of the summands in the Weyl-Kac series restricts the support of the corresponding Weyl group
elements. We combine these ingredients to get the main result, the stable Weyl-Kac formula,
in \cref{sec:stable_formulas}. In \cref{sec:explicit_reduced},
we also describe the corresponding objects explicitly. Further in \cref{sec:explicit_formulas},
we derive some formulas. In the case
$GL$, the Weyl-Kac formula looks as follows.
\begin{theorem}
  Suppose we have two partitions $\mu$ and $\nu$,
  and a nonnegative integer $k$ such that $k\geq \mu_1+\nu_1$.
  Then we have
  \begin{equation*}
    \chr L([\mu,\nu],k) = \sum_{\text{partition}\ \lambda} (-1)^{|\lambda|} q^{\delta(\lambda,\mu,\nu)} \chr M(\lambda\cdot [\mu,\nu],k)
  \end{equation*}
  for any $t\in \kk$ transcendental over $\Q$,
  where the given a partition $\lambda$ with the Frobenius coordinates
  $\lambda = (p_1,\ldots,p_b \mid q_1,\ldots,q_b)$, we define
  \begin{gather*}
    \delta(\lambda,\mu,\nu) = |\lambda| + \sum_{i=1}^b (k - \mu_{q_i+1} -
    \nu_{p_i+1}),\\
    \lambda\cdot[\mu,\nu] = \\
    [
    (k-\nu_{p_1+1},k-\nu_{p_2+1},\ldots,k-\nu_{p_b+1},
    \mu_1,\ldots,\hat{\mu_{q_1+1}},\ldots,\hat{\mu_{q_2+1}},\ldots,)
    + \lambda
    ,\\
    (k-\mu_{q_1+1},k-\mu_{q_2+1},\ldots,k-\mu_{q_b+1},
    \nu_1,\ldots,\hat{\nu_{q_1+1}},\ldots,\hat{\nu_{q_2+1}},\ldots,)
    + \lambda^t
    ].
  \end{gather*}
\end{theorem}

Finally, we put the categorical dimensions into the derived character formula
to get the Nekrasov-Okounkov hook length formula, see \cite[Formula (6.12)]{NO06}, \cite{Han08}.
The analogs of the formula in the other types can be
easily derived as well.

\subsection*{Acknowledgments}

I am grateful to Pavel Etingof for numerous insightful discussions and all the time and the support he provided.
I would also like to thank Leonid Rybnikov for a lot of useful conversations on the paper.

\addtocontents{toc}{\protect\setcounter{tocdepth}{2}}
\section{Preliminaries}
\label{sec:preliminaries}

\subsection{Partitions, Frobenius coordinates, and bipartitions}

By a partition $\lambda$ we mean an infinite tuple $(\lambda_1,\lambda_2,\ldots)$ of
nonnegative integers such that $\lambda_i \geq \lambda_{i+1}$ for any $i$, and
only a finite number of $\lambda_i$ are nonzero. Denote the sum $\sum_i
\lambda_i$ by $|\lambda|$.
The notation $\lambda\vdash n$ means that $n=|\lambda|$.
The length $l(\lambda)$ is the smallest nonnegative integer such that
$\lambda_{l(\lambda)+1} = 0$.

A bipartition $\lambda$ is a pair $[\mu,\nu]$ of partitions.
We set $|\lambda| = |\mu|+|\nu|$, $l(\lambda) = l(\mu)+l(\nu)$. 
The notation $\lambda\vdash (m,n)$ means
that $m=|\mu|,n=|\nu|$.

We denote by $P^{+}_\infty$ the set of all bipartitions in the case $GL$, and
the set of all partitions in the cases $O$ and $Sp$.

There is another way to encode the partitions, the
so-called \emph{Frobenius coordinates}. Being evaluated on a partition, they give a pair of
decreasing nonnegative integer sequences of the same length in the following way.
Suppose we want to compute the Frobenius coordinates of a partition $\lambda$.
Let $b$ be the greatest $i$ such that $\lambda_i \geq i$,
and define the first sequence in the coordinates to be equal to $(\lambda_1-1,\lambda_2-2,\ldots,\lambda_b-b)$.
It is easy to see that we get the same $b$ if we start with $\lambda^t$ instead
of $\lambda$. Define the second sequence in the coordinates to be equal to the
first sequence of $\lambda^t$. We write
\begin{equation*}
  \lambda = (\lambda_1-1,\lambda_2-2,\ldots,\lambda_b-b \mid (\lambda^t)_1-1,(\lambda^t)_2-2,\ldots,(\lambda^t)_b-b).
\end{equation*}
From the definition we see that to transpose a partition in the Frobenius
coordinates we need just to swap the sequences. Also one can clearly argue by
induction to show that the Frobenius coordinates give a bijection between the
set of partitions and the set of pairs of decreasing nonnegative integer
sequences of the same length.

\subsection{Classical linear groups}

We are interested in three families of classical linear groups: $GL_n(\kk)$,
$O_n(\kk)$, and $Sp_n(\kk)$,
where $\kk$ is a field of characterictic $0$.
To treat the cases uniformly, we work with a sequence of groups $G_n$ that
represents one of these families, where $G_n$ denotes the group from a family
with the defining $n$-dimensional representation $V$. In the cases $O$ and $Sp$,
$V$ also carries a nondegenerate bilinear form.
Denote the connected component of the identity in $G_n$ by $G_n^0$,
the Lie algebra of $G_n$ by $\lie{g}_n$,
and the commutator $[\lie{g}_n,\lie{g}_n]$ of $\lie{g}_n$ by $\lie{g}'_n$.
To consider only the situations where $\lie{g}'_n$ is simple, we restrict $n$ to
be an element of $I$, where 
$I = \{ n\in\Z \mid n\geq 2 \}$ in the case $GL$,
$I = \{ n\in \Z \mid n \geq 5 \}$ in the case $O$,
and 
$I = \{ n\in 2\Z \mid n\geq 4 \}$
in the case $Sp$.

Let us describe the basic objects explicitly in the
types $A, B, C$ and $D$.
By $e_1,\ldots,e_n$ we denote a suitable basis of $V$,
and by $r$ we denote the rank of $\lie{g}'_n$.

\begin{center}\bf
Type $A$:
\end{center}

\bgroup
\def\arraystretch{1.2}
\begin{tabular}{>{$}l<{$}}
n=r+1,\\
\text{$V$ has no form,} \\
G_n \iso GL_n(\kk), G_n^0 \iso GL_n(\kk), \lie{g}_n \iso \lie{gl}_n(\kk), \lie{g}'_n \iso \lie{sl}_n(\kk). \\
\end{tabular}
\egroup

\begin{center}\bf
Type $B$:
\end{center}

\bgroup
\def\arraystretch{1.2}
\begin{tabular}{>{$}l<{$}}
n=2r+1,\\
(e_i,e_{r+i})=(e_{r+i},e_i)=1,1\leq i\leq r,(e_{2r+1},e_{2r+1})=1,\ \text{all other pairs give $0$,} \\
G_n \iso O_n(\kk), G_n^0 \iso SO_n(\kk), \lie{g}_n \iso \lie{so}_n(\kk), \lie{g}'_n \iso \lie{so}_n(\kk). \\
\end{tabular}
\egroup

\begin{center}\bf
Type $C$:
\end{center}

\bgroup
\def\arraystretch{1.2}
\begin{tabular}{>{$}l<{$}}
n=2r,\\
(e_i,e_{r+i})=-(e_{r+i},e_i)=1, 1\leq i\leq r,\ \text{all other pairs give $0$,} \\
G_n \iso Sp_n(\kk), G_n^0 \iso Sp_n(\kk), \lie{g}_n \iso \lie{sp}_n(\kk), \lie{g}'_n \iso \lie{sp}_n(\kk). \\
\end{tabular}
\egroup

\begin{center}\bf
Type $D$:
\end{center}

\bgroup
\def\arraystretch{1.2}
\begin{tabular}{>{$}l<{$}}
n=2r,\\
(e_i,e_{r+i})=(e_{r+i},e_i)=1, 1\leq i\leq r,\ \text{all other pairs give $0$,} \\
G_n \iso O_n(\kk), G_n^0 \iso SO_n(\kk),\lie{g}_n \iso \lie{so}_n(\kk), \lie{g}'_n \iso \lie{so}_n(\kk). \\
\end{tabular}
\egroup

\smallskip

Choose a Borel subgroup $B_n \subset G_n^0$, and let $T_n\subset B_n$ be a
maximal torus in $B_n$.
Denote the Lie algebra of $T_n$ by $\lie{h}_n$.
Let $P_n$ be the abelian group of
$T_n$-characters, and $\dom{P}_n\subset P_n$ be the dominant weights corresponding to $B_n$.

Let us
describe explicitly these objects case by case. In what follows, we denote by
$E_{i,j}\in\lie{gl}(V)$ the endomorphism that 
sends $e_k$ to $\delta_{j,k}e_i$. In every case, the Cartan subalgebra
$\lie{h}_n$ is
a subspace of the linear span of $E_{i,i}$, $1\leq i \leq n$. Denote by $\eps_i$ the element
of $\dual{\lie{h}_n}$ that evaluates to the coefficient of $E_{i,i}$.
Finally, denote the Killing form on $\lie{g}_n$, and the induced form on $P_n$, by $(\cdot,\cdot)$. 
The Killing form is rescaled so that $(\eps_i,\eps_j) = \delta_{i,j}$.
Given a set $X$, by $\kk X$ we denote the linear span of $X$.

\begin{center}\bf
Type $A$:
\end{center}

\bgroup
\def\arraystretch{1.2}
\begin{tabular}{>{$}l<{$}}
  n = r+1,\\
  \lie{h}_n = \kk\{ E_{i,i} \mid 1\leq i\leq r+1 \},\\
  P_n = \{ \sum_{i=1}^{r+1} \beta_i\eps_i \mid \forall 1\leq i\leq r+1 \:\beta_i\in\Z  \} \iso \Z^{r+1},\\
  P_n^{+} = \{\beta \mid \beta\in P_n, \beta_1\geq\beta_2\geq\ldots\beta_{r+1}  \}.\\
\end{tabular}
\egroup

\begin{center}\bf
Type $B$:
\end{center}

\bgroup
\def\arraystretch{1.2}
\begin{tabular}{>{$}l<{$}}
  n = 2r+1,\\
  \lie{h}_n = \kk\{ E_{i,i} - E_{r+i,r+i}\mid 1\leq i\leq r \},\\
  P_n = \{ \sum_{i=1}^r \beta_i\eps_i \mid \forall 1\leq i\leq r\: 2\beta_i\in\Z, \forall 1\leq i<j\leq r\: \beta_i-\beta_j\in\Z\}\\
  \quad\iso \Z^r \cup \left(\Z + \frac{1}{2}\right)^{\! r},\\
  P_n^{+} = \{ \beta \mid \beta\in P_n, \beta_1\geq\beta_2\geq\ldots\beta_r\geq 0 \}.\\
\end{tabular}
\egroup

\begin{center}\bf
Type $C$:
\end{center}

\bgroup
\def\arraystretch{1.2}
\begin{tabular}{>{$}l<{$}}
  n = 2r,\\
  \lie{h}_n = \kk\{ E_{i,i} - E_{r+i,r+i}\mid 1\leq i\leq r \},\\
  P_n = \{ \sum_{i=1}^r \beta_i\eps_i \mid \forall  1\leq i\leq r \:\beta_i\in\Z \} \iso \Z^r,\\
  P_n^{+} = \{ \beta \mid \beta\in P_n, \beta_1\geq\beta_2\geq\ldots\beta_r\geq 0 \}.\\
\end{tabular}
\egroup

\begin{center}\bf
Type $D$:
\end{center}

\bgroup
\def\arraystretch{1.2}
\begin{tabular}{>{$}l<{$}}
  n =2r,\\
  \lie{h}_n = \kk\{ E_{i,i} - E_{r+i,r+i}\mid 1\leq i\leq r \},\\
  P_n = \{ \sum_{i=1}^r \beta_i\eps_i \mid \forall 1\leq i\leq r\: 2\beta_i\in\Z, \forall 1\leq i<j\leq r\: \beta_i-\beta_j\in\Z\}\\
  \quad\iso \Z^r \cup \left(\Z + \frac{1}{2}\right)^{\! r},\\
  P_n^{+} = \{ \beta \mid \beta\in P_n, \beta_1\geq\beta_2\geq\ldots\beta_{r-1}\geq|\beta_r| \}.\\
\end{tabular}
\egroup

\medskip

Let $\rep(G_n^0)$ be the category of finite-dimensional $G_n^0$-representations.
The irreducible objects of $\rep(G_n^0)$ are enumerated by $\dom{P}_n$.
Denote the irreducible representation corresponding to $\lambda\in\dom{P}_n$ by
$L_\lambda\in \rep(G_n^0)$.

\subsection{Deligne's category}
\label{sec:deligne_category}

Here we sketch the definition of Deligne's category $\drep(G,R)$,
where $R$ is a commutative $\Q[T]$-algebra.
First, we define the tensor category $\drep^{0}(G,R)$.  

\begin{center}\bf
Case $GL$:
\end{center}

The category $\drep^{0}(GL,R)$ is a free rigid monoidal $R$-linear category
generated by an object $V$ of dimension $T\in R$. See \cite[Section 3.1]{CW2011} for
a more detailed explanation and the diagramatic description of the morphisms in
this category. In particular, the endomorphism algebra of
the mixed tensor power $V^{\tens r} \tens (V^*)^{\tens s}$ is the walled Brauer
algebra $B_{r,s}(T)$ over $R$.

\begin{center}\bf
Case $O$:
\end{center}

The category $\drep^{0}(O,R)$ is a free rigid monoidal $R$-linear category
generated by an object $V$ of dimension $T\in R$ with a symmetric isomorphism
$V\iso V^*$. See \cite[Section 2.1]{CH2015} for
a more detailed explanation and the diagramatic description of the morphisms in
this category. In particular, the endomorphism algebra of the object
$V^{\tens r}$ is the Brauer algebra $B_{r}(T)$ over $R$.

\begin{center}\bf
Case $Sp$:
\end{center}

The category $\drep^{0}(Sp,R)$ is a free rigid monoidal $R$-linear category
generated by an object $V$ of dimension $T\in R$ with an anti-symmetric isomorphism
$V\iso V^*$. It can also be constructed using the $O$ case. First, we take
the category
$\drep^{0}(O,\tilde{R})$, where $\tilde{R}$ is the ring $R$ with the $\Q[T]$-algebra
structure twisted by $T \mapsto -T$.
Then we multiply the symmetric braiding morphism $c\: V^{\tens r}\tens V^{\tens r'} \to
V^{\tens r'}\tens V^{\tens r}$ by $(-1)^{rr'}$. 
For more details see
\cite[Section 9.5]{DeligneSecond} and
\cite[Section 2.6]{EtingofII}.

\

\emph{Deligne's category $\drep(G,R)$} is the
Karoubi closure of the additive closure of the category $\drep^{0}(G,R)$. 
From the universal property of a Karuobi closure, it follows that for any
commutative $R$-algebra $S$ we have a linear tensor
functor $\tens S\: \drep(G,R)\to\drep(G,S)$.

For $t\in\kk$, denote by $\kk_t$ the field $\kk$ with the $\Q[T]$-algebra
structure sending $T$ to $t$.
Putting $\kk_t$ into Deligne's category construction, we get a $\kk$-linear
category $\drep(G,\kk_t)$ which we denote by $\drep(G_t)$.
We proceed to describe the indecomposable objects of $\drep(G_t)$.
In what follows, we also suppose that $t\neq 0$ for the sake of simplicity, as
this case can be safely dropped for our purposes.
Denote the symmetric group on $k$ elements by $\Sigma_k$.
Recall that $P^{+}_\infty$ denotes the set of all bipartitions in the case $GL$, and
the set of all partitions in the cases $O$ and $Sp$.

\begin{center}\bf
Case $GL$:
\end{center}

\begin{proposition}
  \label{stm:gl_brauer_step}
  There exists an idempotent $e_{r,s}\in B_{r,s}(t)$ such that
  $B_{r,s}(t)/(e_{r,s}) \iso \kk[\Sigma_r\times\Sigma_s]$ and
  $B_{r-1,s-1}(t) \iso e_{r,s}B_{r,s}(t)e_{r,s}$. The latter isomorphism
  is $x \mapsto \frac{1}{t}\psi_{r,s}x\hat{\psi}_{r,s}$ for some diagrams
  $\psi_{r,s}\in Hom(V^{\tens r-1}\tens (V^*)^{\tens s-1}, V^{\tens r}\tens (V^*)^{\tens s})$
  and
  $\hat{\psi}_{r,s}\in Hom(V^{\tens r}\tens (V^*)^{\tens s}, V^{\tens r-1}\tens (V^*)^{\tens s-1})$.
\end{proposition}

\begin{proof}
  See \cite[Proposition 2.1, Proposition 2.3 and (2)]{Cox2008} for the two
  isomorphisms, and see \cite[Section 4.4]{CW2011} for the diagrams $\psi_{r,s}$
  and $\hat{\psi}_{r,s}$.
\end{proof}

For any bipartition $\lambda\vdash (r,s)$, pick a primitive idempotent
$z_\lambda\in \kk[\Sigma_r\times\Sigma_s]$ in the corresponding conjugacy class.

\begin{proposition}
  \label{stm:gl_brauer_down}
  For any bipartition $\lambda\vdash (r,s)$, there exists a primitive idempotent
  $e_\lambda\in B_{r,s}(t)$ such that the image of $e_\lambda$ after the
  factorization by $(e_{r,s})$ is $z_\lambda$. 
\end{proposition}

\begin{proof}
  See \cite[Section 4.3]{CW2011}.
\end{proof}

\begin{theorem}
  \label{stm:gl_brauer_enum}
  The indecomposable objects of $\drep(G_t)$ are parametrized by the set
  $\dom{P}_\infty$ of all bipartitions, with the object $L_\lambda = (V^{\tens
    r}\tens (V^*)^{\tens s}, e_\lambda)$ corresponding to a bipartition $\lambda\vdash (r,s)$.
  The set $\{\lambda \mid l\geq 0, \lambda \vdash (r-l,s-l)\}$ of bipartitions enumerates the
  isomorphism classes of indecomposable objects coming from the object
  $V^{\tens r}\tens(V^*)^{\tens s}$.
\end{theorem}

\begin{proof}
  Given the fact that the objects $L_\lambda$ are not isomorphic for different
  $\lambda$, this is merely a rephrasing of \cref{stm:gl_brauer_step} and
  \cref{stm:gl_brauer_down}.
  See \cite[Theorem 2.7]{Cox2008}, \cite[Theorem 4.5.1]{CW2011}. 
\end{proof}

\begin{center}\bf
Cases $O$ and $Sp$:
\end{center}

\begin{proposition}
  \label{stm:o_brauer_step}
  There exists an idempotent $e_r\in B_r(t)$ such that
  $B_r(t)/(e_r) \iso \kk[\Sigma_r]$ and
  $B_{r-2}(t)\iso e_rB_r(t)e_r $. The latter isomorphism
  is $x \mapsto \frac{1}{t}\psi_{r}x\hat{\psi}_{r}$ for some diagrams
  $\psi_{r}\in Hom(V^{\tens r-2}, V^{\tens r})$
  and
  $\hat{\psi}_{r}\in Hom(V^{\tens r}, V^{\tens r-2})$.
\end{proposition}

\begin{proof}
  See \cite[Lemma 2.1 and (2.1)]{Cox2009} for the two isomorphisms. The diagrams
  $\psi_r$ and $\hat{\psi}_r$ are just $\psi_{r-1,1}$ and $\hat{\psi}_{r-1,1}$
  from the case $GL$ after the identification of $V$ and $V^*$.
\end{proof}

For any partition $\lambda\vdash r$, pick a primitive idempotent
$z_\lambda\in \kk[\Sigma_r]$ in the corresponding conjugacy class.

\begin{proposition}
  \label{stm:o_brauer_down}
  For any partition $\lambda\vdash r$, there exists a primitive idempotent
  $e_\lambda\in B_r(t)$ such that the image of $e_\lambda$ after the
  factorization by $(e_r)$ is $z_\lambda$. 
\end{proposition}

\begin{proof}
  See \cite[Section 3.2]{CH2015}.
\end{proof}

\begin{theorem}
  \label{stm:o_brauer_enum}
  The indecomposable objects of $\drep(G_t)$ are parametrized by the set
  $\dom{P}_\infty$ of all partitions, with the object $L_\lambda = (V^{\tens r},
  e_\lambda)$ corresponding to a partition $\lambda\vdash r$.
  The set $\{\lambda \mid l\geq 0, \lambda \vdash r-2l\}$ of partitions enumerates the
  isomorphism classes of indecomposable objects coming from the object
  $V^{\tens r}$.
\end{theorem}

\begin{proof}
  Given the fact that the objects $L_\lambda$ are not isomorphic for different
  $\lambda$, this is merely a rephrasing of \cref{stm:o_brauer_step} and \cref{stm:o_brauer_down}.
  See \cite[Section 2, p. 277]{Cox2009}, \cite[Theorem 3.4, Theorem 3.5]{CW2011}. 
\end{proof}

When $t$ is not an integer, Deligne's category $\drep(G_t)$ is in fact
abelian and semisimple, see \cite[Theorem 4.8.1]{CW2011} for the case $GL$, and
\cite[Theorem 3.6]{CH2015} for the cases $O$ and $Sp$.

\subsection{Filtrations on $\rep(G_n)$ and Deligne's category}

Let $\filt{P^{+}}{m}\subset P^{+}_{\infty}$ be the subset of all (bi)partitions
$\lambda$ such that $|\lambda|\leq m$.
Also consider a polynomial $\bad_m\in\Z[T]$ which is equal to $\prod_{n=-m}^{m}
(T-n)$ in the case $GL$, and to $\prod_{n=-2m-2}^{2m+2} (T-n)$ in the cases $O$ and
$Sp$. Fix any $m\in\Z_{\geq 0}$ and $n\in I\setminus Z(\bad_m)$.

A type by type inspection gives that the set $\filt{P^{+}}{m}$ naturally embeds to $\dom{P}_n$.
Indeed, in the case $GL$ send a bipartition $[\mu,\nu]\in\filt{P^{+}}{m}$ to the weight
$\sum_i \mu_i\eps_i - \sum_j \nu_j \eps_{n+1-j}\in\dom{P}_n$, and in the
cases $O$ and $Sp$ send a partition $\lambda\in\filt{P^{+}}{m}$ to
$\sum_i \lambda_i\eps_i\in\dom{P}_n$. It is easy to
check that the resulting weights are dominant.

\begin{definition}
Let $\filt{\rep(G_n^0)}{m}$ be a full subcategory of $\rep(G^0_n)$ containing all
the objects $X$ such that every simple subobject of $X$ is isomorphic to
$L_{\lambda}$ for some $\lambda\in\filt{\dom{P}}{m}$.
\end{definition}

In Deligne's category case, we define the filtration as in
\cite[Proposition 9.8]{DeligneSecond} and
\cite[Proposition 10.6]{DeligneSecond}.

\begin{definition}
  For a commutative $\Q[T]$-algebra $R$,
  let $\filt{\drep(G,R)}{m}$ be the full subcategory of objects in $\drep(G,R)$
  coming from the objects $V^{\tens r} \tens (V^*)^{\tens s}\in\drep^0(G,R)$ with $r+s\leq m$.
\end{definition}

Due to \cref{stm:gl_brauer_enum,stm:o_brauer_enum}, in the particular case
$R=\kk_t$ the filtration $\filt{\drep(G_t)}{m}$ is the full additive subcategory
of $\drep(G_t)$ generated by the objects $L_{\lambda}$, $\lambda\in \dom{P}_m$.

It turns out that these filtrations compare nicely.
When $t$ is equal to $n\in I$, the universal property of Deligne's category
implies the existence of a linear tensor functor $F_n\:\drep(G_n)\to \rep(G^0_n)$
that sends $V\in\drep(G_n)$ into $V\in\rep(G^0_n)$. 

\begin{theorem}
  \label{stm:deligne_finite_connection}
  For any $m\in\Z_{\geq 0}$, $n\in I\setminus Z(E_m)$, and $\lambda\in\filt{\dom{P}}{m}$,
  we have $F_n(L_{\lambda}) \iso L_{\lambda}\in\rep(G^0_n)$.
\end{theorem}

\begin{proof}
  The case $GL$ is described in \cite[Theorem 5.2.2]{CW2011}, and the case
  $Sp$ is discussed in \cite[Section 7.1]{CH2015}. We need some extra care in
  the case $O$. It is proved in \cite[Proposition 5.1]{CH2015} that the image of
  $L_{\lambda}\in \drep(G_n)$ in $\rep(G_n)$ is isomorphic to
  $\mathbb{S}_{[\lambda]}V$, which is an irreducible $O_n(\kk)$-representation
  defined in \cite[Section 19.5]{Fulton2004}. In our case, the last
  coordinate of the highest weight corresponding to $\lambda$ is zero.
  Therefore, \cite[Theorem 19.22]{Fulton2004} says that $F_n(L_{\lambda})$,
  which is just $\mathbb{S}_{[\lambda]}V$ as an $SO_n(\kk)$-representation, is
  the irreducible representation with the highest weight $\lambda$. Claim follows.
\end{proof}

\subsection{The affine Lie algebra in a tensor category and its BGG category}
\label{sec:affine_bgg}

This subsection recalls some definitions from \cite[Section 6]{EtingofII} that
we need.

Suppose we have a symmetric monoidal rigid category $\cat{C}$, and a Lie algebra
object $\lie{g}\in\cat{C}$ in it such that each object of $\cat{C}$ has a distinguished structure of a
$\lie{g}$-module in a functorial way. 
Let $\triv$ be the tensor identity object in $\cat{C}$.

\begin{definition}
  Let $\cat{C}^a$ be a category of the objects in $\cat{C}$ graded in the
  two-dimensional lattice $\Z\Lambda_0\oplus\Z\delta$, possibly with infinite
  number of gradings.
\end{definition}

Consider the Lie algebra $\lie{l} =
\lie{g}\oplus\triv\oplus\triv'\in \cat{C}^a$ of grading $0$, where $\triv'$
is just a different notation for the trivial object. Define the action of
$\lie{l}$ on an object $X\in\cat{C}^a$ of grading $k\Lambda_0+j\delta$ as follows. The
action of $\lie{g}\subset\lie{l}$ is the one coming from the action of $\lie{g}$
on $X$ as an object of $\cat{C}$. The action of $\triv\subset\lie{l}$ is the
multiplication by $k$, and the action of $\triv'\subset\lie{l}$ is the
multiplication by $j$. This action obviously extends to the functorial action on
$\cat{C}^a$. 

We can define the Killing symmetric bilinear form
$Kil\:\lie{g}\tens\lie{g}\to\triv$, see \cite[Section 6]{EtingofII} for details.
Also denote the trivial object $\triv$ of grading $\delta$ by $z\in\cat{C}^a$.

\begin{definition}
  Define \emph{the affine Lie algebra} $\aff{\lie{g}}$ of $\lie{g}$ to be the
  object $\left(\bigoplus_{i<0} \lie{g}z^i\right) \oplus \lie{l}
  \oplus \left(\bigoplus_{i>0} \lie{g}z^i\right) \in \cat{C}^a$ with the
  following Lie algebra structure.
The inner action of $\lie{l}$ corresponds to the action of $\lie{l}$ on
$\cat{C}^a$, and the commutator between any $\lie{g}z^m$ and $\lie{g}z^n$ is just
the usual commutator $[,]\:\lie{g}z^m\tens\lie{g}z^n\to \lie{g}z^{m+n}$ plus the
$2$-cocycle $m\delta_{m,-n}Kil$ going to $\triv\subset\lie{l}$.
\end{definition}

Note that $\aff{\lie{g}}$ comes with a (parabolic) triangular decomposition
$\aff{\lie{g}} = \lie{u}^{-}\oplus\lie{l}\oplus\lie{u}^{+}$, where
$\lie{u}^{-}$ is equal to $\bigoplus_{i<0} \lie{g}z^{i}$,
and 
$\lie{u}^{+}$ is equal to $\bigoplus_{i>0} \lie{g}z^i$.

Our next goal is to define the parabolic BGG category of $\aff{\lie{g}}$.
First, by a representation of $\aff{\lie{g}}$ we mean an object in
$\cat{C}^a$ with an action of $\aff{\lie{g}}$ compatible with the
action of $\lie{l}\subset\aff{\lie{g}}$.

\begin{definition}
\emph{The parabolic BGG category} $\aff{\cat{O}}$ is the full subcategory of $\aff{\lie{g}}$-modules that are finitely
generated as $\lie{u}^{-}$-modules and locally nilpotent as $\lie{u}^{+}$-modules.
\end{definition}

Denote the additive Grothendieck group of an arbitrary category $\cat{D}$
by $K(\cat{D})$.
Note that $K(\cat{C}^a)$ is isomorphic to the space of set morphisms $\Hom(\Z\Lambda_0\oplus\Z\delta,
K(\cat{C}))$.
By a formal character $\chr M$ of an object $M\in
\aff{\cat{O}}$ we mean the class $[M]\in K(\cat{C}^a)$ of $M$ as an object in $\cat{C}^a$.

Now suppose that $\cat{C}$ is a semisimple abelian category.
Then both $\cat{C}^a$ and $\aff{\cat{O}}$ are abelian.
Denote the set of isomorphism classes of simple
objects in $\cat{C}$ by $\dom{P}$, and denote the simple object
corresponding to $\lambda\in\dom{P}$ by $L_\lambda$. Then the simple objects of
$\cat{C}^a$ are parametrized by $\dom{P}\times\Z\Lambda_0\times\Z\delta$.
Denote the object corresponding to $\phi\in
\dom{P}\times\Z\Lambda_0\times\Z\delta$ by $L_\phi$. 
Consider some examples of the objects in $\aff{\cat{O}}$.

\begin{definition}
For any simple object $L_{\phi}\in\cat{C}^a$,
define \emph{the parabolic Verma module} $M(\phi)$ to be
the tensor product $U(\aff{\lie{g}})\tens_{U(\lie{l}\oplus\lie{u}^{+})} L_{\phi}$,
where $\lie{u}^{+}$ acts on $L_{\phi}$ by $0$.
\end{definition}

\begin{lemma}
  $M(\phi)$ has a unique simple quotient $L(\phi)$.
\end{lemma}

\begin{proof}
  Let $\phi = \lambda + k\Lambda_0 + j\delta$.
  Any proper subobject of $M(\phi)$ has only gradings $k\Lambda_0 +
  (j-j')\delta$, where $j' > 0$. Therefore, the sum of all proper submodules is
  still a proper submodule, and is the maximal proper submodule. The
  corresponding quotient is a unique simple quotient.
\end{proof}

The first example of this setting is the finite case of
$\rep(G_n^0)$. We consider $\lie{g}'_n$ as a Lie algebra object in $\rep(G_n^0)$.
We indeed have a distinguished $\lie{g}'_n$-module structure on each object of $\rep(G_n^0)$.
Denote the corresponding affine Lie algebra by $\aff{\lie{g}'}_n$. 

The second situation we are interested in is Deligne's category with a non-integer $t$.
Let $\lie{g}_t$ be the Lie algebra object in $\drep(G_t)$ equal to
$V\tens V^*$ in the case $GL$, to $L_{(1,1)}$ in the case $O$,
and to $L_{(2)}$ in the case $Sp$. Denote its commutator by $\lie{g}'_t$, which
only differs from $\lie{g}_t$ in the case $GL$ where it is isomorphic to
$L_{(1),(1)}$. It is also a Lie algebra object in $\drep(G_t)$ such that each
object in $\drep(G_t)$ has a distinguished structure of $\lie{g}'_t$-module.
Denote the corresponding affine Lie algebra by $\aff{\lie{g}'}_t$.

\subsection{Standard facts about roots and Weyl group}
\label{sec:bourbaki_stuff}

This section basically restates some propositions from \cite{Bourbaki46}. Althrough the
propositions we need are proved
only for finite diagrams in Bourbaki's book, the proofs work literally for all
the cases we need. 

For an arbitrary Dynkin diagram $X$,
denote by $\Delta(X)$, $\Delta^{\vee}(X)$, $Q(X)$, $Q^{\vee}(X)$, and $W(X)$ the
root system, the coroot system, the root lattice, the coroot lattice, and
the Weyl group of $X$, resp.
For this subsection, fix a diagram $X$ and let us omit the reference to $X$ in
the notations of these objects.
Pick a subset $\Pi = \{ \alpha_x \in \Delta \mid x\in X \}$ of simple roots.
Denote the reflection in a simple root $\alpha_x$, $x\in X$, by $s_x$.
Let $\Delta^{+}$
(resp. $\Delta^{-}$) be the 
corresponding positive (resp. negative) roots.

\begin{definition}
  Take an element $w\in W$. Denote the set $\Delta^{+}\cap w(\Delta^{-})$ by
  $I(w)$. We will call it \emph{the set of inversion roots of $w$}.
\end{definition}

\begin{lemma}
  \label{stm:full_description_of_negatives}
  For any reduced expression $w = s_{x_1}s_{x_2}\ldots s_{x_j}$, we have $|I(w)| = j$
  and $I(w) = \{ s_{x_1}\ldots s_{x_{i-1}} (\alpha_{x_i}) \mid 1\leq i
  \leq j  \}$.
\end{lemma}

As a simple corollary, we have the following lemma.

\begin{lemma}
  \label{stm:simple_reflection_permutes_almost}
  Any simple reflection $s_x\in W$ permutes
  the set $\Delta^{+}\setminus \{\alpha_x\}$.
\end{lemma}

Suppose we have a weight lattice $P$ of $X$, i.e. a $\Z$-module $P$
with a $\Z$-linear embedding $Q\to P$ and a $\Z$-linear evaluation map
$P\tens_\Z Q^{\vee}\to \Z$ such that the composite pairing $Q\tens_{\Z} Q^{\vee}
\to \Z$ of
$Q$ and $Q^{\vee}$ is the default one.
Note that the Weyl group $W$ indeed acts on any weight lattice
$P$.
Denote the set of dominant weights, i.\,e. the elements of $P$ where each
simple coroot evaluates non-negatively, by $\dom{P}$.

\begin{lemma}
  \label{stm:weyl_take_dominant_lower}
  For any dominant weight $\phi\in\dom{P}$ and Weyl group element $w\in W$, the difference
  $\phi-w\phi$ can be expressed as $\sum_{x\in X} k_x\alpha_x$, where $k_x\geq
  0$ for all $x\in X$.
\end{lemma}

On $P$, we can also define the dot action by setting
$s_x \cdot \phi = \phi - (\langle \phi, \alpha_x^{\vee}\rangle +
1)\alpha_x$, $\phi\in P$, $x\in X$. It is easy to check that this is
indeed an action of $W$ on $P$. 

\begin{lemma}
  \label{stm:dot_action_by_negatives_sum}
  Given any $w\in W$, we have
  $$
  w\cdot 0 = -\sum_{\alpha\in I(w)} \alpha.
  $$
\end{lemma}

\begin{lemma}
  \label{stm:dot_action_on_zero_is_unique}
  The dot action of an element $w\in W$ on $0$ determines $w$ uniquely.
\end{lemma}

Let $X$ be a subdiagram of some other Dynkin diagram $\aff{X}$.
We mark all the objects corresponding to $\aff{X}$ with a hat, for example,
denote the set of roots by $\aff{\Delta}$.
We do not require $\aff{X}$ to be the affinization of $X$.

\begin{definition}
  Say that a Weyl group
  element $w\in \aff{W}$ is \emph{$X$-reduced} if the set $I(w)$
  of the inversion roots of $w$ does not intersect with $\Delta$. Denote the subset of $X$-reduced
  elements by $\wx\subset \aff{W}$.
\end{definition}

\begin{definition}
  Say that a weight $\phi\in\aff{P}$ is \emph{$X$-dominant} if $\langle \phi,
  \alpha^{\vee}_x\rangle \geq 0$ for any $x\in X$. Denote the set of
  $X$-dominant weights by $\xdom{P}$.
\end{definition}

\begin{lemma}
  \label{stm:wx_gives_x_dominant}
  For any dominant $\phi\in\dom{\aff{P}}$ and $X$-reduced $w\in\wx$, the weight
  $w\cdot\phi$ is $X$-dominant.
\end{lemma}

\begin{lemma}
  \label{stm:wx_wx_gives_all}
  The map $W\times \wx \to \aff{W},$
  $$
  (w, w')\in W\times \wx \mapsto ww'\in \aff{W}
  $$
  is a bijection.
\end{lemma}

\subsection{Data of the usual and the affine Lie algebra}

The main reference for this section is \cite[Section 6.2]{Kac1983}.

Let $X_n$ be the Dynkin diagram of $\lie{g}'_n$. Mark the corresponding objects
with a subscript $n$. For example, let $\Delta_n\subset P_n$ be the set of roots
of $X_n$. Additionally, let $\theta\in\Delta^{+}_n$ be the highest root.
The description of these objects in each type follows.

\begin{center}\bf
Type $A$:
\end{center}

\bgroup
\def\arraystretch{1.2}
\begin{tabular}{>{$}l<{$}}
  n = r+1,\\
  \Delta^{+}_n = \{\eps_i-\eps_j \mid 1\leq i<j\leq r+1  \},\\
  \Pi_n = \{\alpha_i := \eps_i-\eps_{i+1}\mid  1\leq i\leq r \},\\
  \Pi^{\vee}_n = \{ \alpha_i^{\vee} := E_{i,i} - E_{i+1,i+1}\mid  1\leq i\leq r\},\\
  \theta = \eps_1-\eps_{r+1},\  \theta^{\vee} = E_{1,1}-E_{r+1,r+1}, \\
\end{tabular}
\egroup

\begin{center}\bf
Type $B$:
\end{center}

\bgroup
\def\arraystretch{1.2}
\begin{tabular}{>{$}l<{$}}
  n = 2r+1,\\
  \Delta^{+}_n = \{ \eps_i-\eps_j, \eps_i+\eps_j
  \mid 1\leq i<j\leq r  \} \cup \{ \eps_i \mid 1\leq i\leq r \},\\
  \Pi_n = \{ \alpha_i := \eps_i-\eps_{i+1} \mid 1\leq i\leq r-1 \} \cup \{ \alpha_r := \eps_r \}, \\
  \Pi_n^{\vee} = \{\alpha_i^{\vee} := E_{i,i} - E_{i+1,i+1}-E_{r+i,r+i} +E_{r+i+1,r+i+1} \mid 1\leq i\leq r-1 \}\\
  \quad\quad\ \ \cup \{ \alpha_r^{\vee} := 2E_{r,r} - 2E_{2r,2r} \},\\
  \theta =  \eps_1+\eps_2, \ \theta^{\vee} = E_{1,1} + E_{2,2}-E_{r+1,r+1} -E_{r+2,r+2} , \\
\end{tabular}
\egroup

\begin{center}\bf
Type $C$:
\end{center}

\bgroup
\def\arraystretch{1.2}
\begin{tabular}{>{$}l<{$}}
  n = 2r,\\
  \Delta^{+}_n = \{ \eps_i-\eps_j,  \eps_i+\eps_j \mid 1\leq i<j\leq r  \}
  \cup \{  2\eps_i \mid 1\leq i\leq r \},\\
  \Pi_n = \{\alpha_i :=  \eps_i-\eps_{i+1} \mid 1\leq i\leq r-1 \} \cup \{ \alpha_r :=  2\eps_r \},\\
  \Pi_n^{\vee} = \{ \alpha_i^{\vee} := E_{i,i} - E_{i+1,i+1}-E_{r+i,r+i} +E_{r+i+1,r+i+1} \mid 1\leq i\leq r-1\}\\
  \quad\quad\ \ \cup \{\alpha_r^{\vee} := E_{r,r} - E_{2r,2r} \},\\
  \theta =  2\eps_1,\ \theta^{\vee} =  E_{1,1} - E_{r+1,r+1},\\
\end{tabular}
\egroup

\begin{center}\bf
Type $D$:
\end{center}

\bgroup
\def\arraystretch{1.2}
\begin{tabular}{>{$}l<{$}}
  n = 2r,\\
  \Delta^{+}_n = \{\eps_i-\eps_j,  \eps_i+\eps_j \mid 1\leq i<j\leq r  \},\\
  \Pi_n = \{ \alpha_i :=  \eps_i-\eps_{i+1} \mid 1\leq i\leq r-1\} \cup \{ \alpha_r :=  \eps_{r-1}+\eps_r \},\\
  \Pi_n^{\vee} = \{ \alpha_i^{\vee} := E_{i,i} - E_{i+1,i+1}-E_{r+i,r+i} +E_{r+i+1,r+i+1} \mid 1\leq i\leq r-1\}\cup \\
  \quad\quad\ \ \cup \{\alpha_r^{\vee} := E_{r-1,r-1} + E_{r,r}-E_{2r-1,2r-1} -E_{2r,2r} \},\\
  \theta = \eps_1+\eps_2,\ \theta^{\vee} = E_{1,1} + E_{2,2}-E_{r+1,r+1} -E_{r+2,r+2},\\
\end{tabular}
\egroup
\smallskip

Now we want to describe the objects of $\aff{\lie{g}'}_n$. Since the weight spaces
are related to $G^0_n$ rather than to $\lie{g}'_n$, we also need to consider
$\aff{\lie{g}}_n$.

We have a subalgebra
$\lie{l}_n = \lie{h}_n\oplus\kk c\oplus\kk d\subset \aff{\lie{g}}_n$ of the
affine algebra $\aff{\lie{g}}_n$. Choose the elements
$\Lambda_0$ and $\delta$ in the dual space $\dual{\lie{l}_n}$ such 
that $\delta\vert_{\lie{h}\oplus\kk c} = 0, \langle \delta, d \rangle = 1,
\Lambda_0\vert_{\lie{h}\oplus\kk d} = 0, \langle\Lambda_0, c\rangle = \frac{(\theta,\theta)}{2}$.
We now have $\lie{l}^*_n = \dual{\lie{h}_n}\oplus\kk \Lambda_0\oplus\kk \delta$.

Denote the Dynkin diagram of $\aff{\lie{g}'}_n$ by $\aff{X}_n$.
Mark the related affine objects with a subscript $n$, for example, denote
the affine Weyl group of $\aff{\lie{g}'}_n$ by $\aff{W}_n$.

\begin{definition}
  Define the weight lattice $\aff{P}_n$ of $\aff{\lie{g}'}_n$ to be
  $P_n\oplus \Z\Lambda_0\oplus \Z\delta\subset \lie{l}^*_n$.
\end{definition}

The other objects are as follows.

\smallskip
\bgroup
\def\arraystretch{1.2}
\begin{tabular}{>{$}l<{$}}
  \aff{\Delta}^{+}_n = \Delta_n^{+}\cup \{\alpha +j\delta
\mid \alpha\in \Delta_n\cup\{0\}, 0<j\in\Z\},\\
  \aff{\Pi}_n = \Pi_n \cup \{\alpha_0
  := \delta-\theta \},\\
  \aff{\Pi}^{\vee}_n = \Pi^{\vee}_n \cup \{\alpha^{\vee}_0
  := \frac{2}{(\theta,\theta)}c-\theta^{\vee} \},\\
\end{tabular}
\egroup

\smallskip

Now we can derive a few properties of dominance. First, it is clear that a
weight $\phi = \lambda + k\Lambda_0 + j\delta$ is $X$-dominant if and only if
its component $\lambda$ is dominant, i.\,e. lies in $\dom{P}_n$.
Moreover, if $\lambda$ in fact lies $\filt{\dom{P}}{m}$, and $n\in I\setminus Z(E_m)$, the sufficient and
necessary condition for the dominance of $\phi$
is easy to establish.

\medskip
\bgroup
\def\arraystretch{1.2}
\begin{tabular}{ll}
  Case $GL$: & $\mu_1+\nu_1\leq k$, where $\lambda = [\mu,\nu]$,\\
  Case $O$: & $\lambda_1+\lambda_2 \leq k$,\\
  Case $Sp$: & $\lambda_1\leq k$.\\
\end{tabular}
\egroup

\smallskip

It can be proved by an examination of the value of $\alpha^{\vee}_0$ on $\phi$.

\subsection{Weyl-Kac formula and Garland formula in the finite case}

Now we proceed to the Weyl-Kac formula for irreducible representations of $\aff{\lie{g}'}_n$.
In the usual form, it expresses the characters as formal sums of the symbols
$e^{\phi}$, $\phi\in \aff{P}_n$.

\begin{theorem}
  \label{stm:usual_finite_kac}
  For a dominant weight $\phi\in\dom{\aff{P}}_n$, we have
  \begin{equation*}
    \chr L(\phi) =
    \frac%
    {\sum_{w\in \aff{W}_n} (-1)^{l(w)} e^{w\cdot \phi}}%
    {\prod_{\alpha\in\aff{\Delta}^{+}_n} (1 - e^{-\alpha})^{\mathrm{mult}\,\alpha}},
  \end{equation*}
  where $\mathrm{mult}\,\alpha$ denotes the dimension of the eigenspace of
  $\alpha$ in $\aff{\lie{g}'}_n$.
\end{theorem}

\begin{proof}
  See \cite[Theorem 10.4]{Kac1983}.
\end{proof}

This form is not suitable for our purposes, since it uses a more refined version
of characters, and uses some expressions that do not have a well-defined limit.
Luckily, it is well-known that one can deduce the parabolic version from the
usual one. We show the deduction here to illustrate how the set $\wx_n$ arises.

\begin{theorem}
  \label{stm:relative_finite_kac}
  For a dominant weight $\phi\in\dom{\aff{P}}_n$, we have
  \begin{equation*}
    \chr L(\phi) =
    \sum_{w\in \wx_n} (-1)^{l(w)}\, \chr M(w\cdot\phi)
  \end{equation*}
\end{theorem}

\begin{proof}
  Pick $w\in\wx_n$.
  By \cref{stm:wx_gives_x_dominant}, the weight $w\cdot\phi$ is $X$-dominant.
  Consider the subsum
  $$
    \frac%
    {\sum_{w'\in W_n} (-1)^{l(w'w)} e^{w'\cdot (w\cdot \phi)}}%
    {\prod_{\alpha\in\aff{\Delta}^{+}_n} (1 - e^{-\alpha})^{\mathrm{mult}\,\alpha}}
  $$ in \cref{stm:usual_finite_kac}.
  By the usual Weyl character formula for $\lie{g}'_n$, it is equal to $\chr U(\lie{u}^{-}_n)
  \cdot (-1)^{l(w)} \chr L_{w\cdot\phi} = (-1)^{l(w)} \chr M(w\cdot\phi)$.
  By \cref{stm:wx_wx_gives_all} we are done.
\end{proof}

Another fact that is closely related to the Weyl-Kac formula and can be seen as a
generalization of the Weyl denominator formula is the Garland formula.

\begin{theorem}
  For any nonnegative $i\in\Z_{\geq 0}$, the $i$-th cohomology of $\lie{u}^{-}_n$
  is isomorphic to 
  $$
  H^i(\lie{u}^{-}_n) \iso \bigoplus_{\substack{w\in \wx_{n} \\ l(w) = i}}
  L_{w\cdot 0}
  $$
  as an $\lie{l}_{n}$-module.
\end{theorem}

\begin{proof}
  See \cite[Theorem 3.2]{Garland1975}.
\end{proof}

\section{Stable Weyl-Kac formula}

\subsection{$M(\phi)$ and $L(\phi)$ coherence}
\label{sec:m_l_coherence}

\newcommand{\str}[1]{\mathrm{#1}}

We want to relate the constructions of $M(\phi)$, $L(\phi)$, and $H^i(\lie{u}^{-})$
in $\drep(G_t)$ for non-integer $t$ and in $\Rep(G_n)$ for $n\in I$. To do that,
we extend this constructions to a particular sheaf of abelian categories over $\kk[T]$.

In \cref{sec:deligne_category} we discussed the category $\drep^0(G,R)$.
Consider instead the same free category, but without any $R$-linearity condition,
and any restrictions on the dimesion of the object $V$.
The resulting corresponding category $\str{Br}$ is indeed
the category of walled Brauer diagrams in the $GL$ case, and of Brauer diagrams
in the $O$ and $Sp$ cases. Note that $\str{Br}$ is
a symmetric monoidal category. Consider the category $\Hom(\str{Br}^{op}, \kk\text{-}\str{mod})$
of $\kk$-linear $\str{Br}$ preseaves. Via the Day convolution, we extend the symmetric
monoidal structure from $\str{Br}$ to $\Hom(\str{Br}^{op}, \kk\text{-}\str{mod})$.
The identity object in $\str{Br}$, the empty set, has the endomorphism set $\Z_{\geq 0}$,
where $m\in \Z_{\geq 0}$ corresponds to the loop in the $m$-th power. Translating to
$\Hom(\str{Br}^{op}, \kk\text{-}\str{mod})$, we get that $\End(\triv) \iso \kk[T]$,
where $T$ represents the loop. Every object bears an action of $\End(\triv)$, so in fact we can
think about $\Hom(\str{Br}^{op}, \kk\text{-}\str{mod})$ as a particular full subcategory
of $\Hom(\str{Br}^{op}, \kk[T]\text{-}\str{mod})$. To emphasize this fact, we denote
$\Hom(\str{Br}^{op}, \kk\text{-}\str{mod})$ by $\drep^{c}(G,\kk[T])$. In a similar manner one can define
$\drep^{c}(G,R)$ for any $\kk[T]$-algebra $R$.

\begin{proposition}
  For any $\kk[T]$-algebra $R$, we have the following:
  \begin{enumerate}
  \item $\drep^{c}(G,R)$ is a cocomplete abelian category,
  \item there is a natural fully faithful embedding
    $\drep(G,R) \to \drep^c(G,R)$.
  \end{enumerate}
\end{proposition}

\begin{proof}
  The first point is obvious from the fact that $\drep^c(G,R)$ is a presheaf category.
  For the second point, note that there is a natural Yoneda embedding $\drep^0(G,R) \to \drep^c(G,R)$.
  Since $\drep^c(G,R)$ is Karoubi complete, the embedding extends to a functor $\drep(G,R)\to \drep^c(G,R)$.
  It is easy to see that the functor gives a fully faithful embedding.
\end{proof}

The categories $\drep^c(G,R)$ form a sheaf of abelian cocomplete categories over $\kk[T]$.
This sheaf provides a good framework for producing constructions that are coherent for different $t$ and $n$.

\begin{proposition}
  Suppose that we have a section $X\in \drep^c(G,P^{-1}\kk[T])$,
  where $P\in\kk[T]$, and $X\tens\kk_t\in \filt{\drep(G_t)}{m}$
  for some fixed $m$ and all $t\in \kk\setminus Z(P)$. Then the class
  $[ X\tens \kk_t ]\in K_m$ is the same for all $t\in \kk\setminus Z(PQ)$,
  where $Q\in\kk[T]$ is a nonzero polynomial.
\end{proposition}

\begin{proof}
  Consider an irreducible object $L_\phi\in\drep(G,\kk(t))$ for $\phi\in\filt{\dom{P}}{m}$.
  The corresponding idempotent requires the localization of a finite number of polynomials.
  Therefore, $L_{\phi}$ can be defined as a section over $Q_\phi^{-1}\kk[T]$ for some
  nonzero $Q_{\phi}\in\kk[T]$. Now take the $\Hom$ space $\Hom(L_\phi,X)$. Since both
  $L_{\phi}$ and $X$ are sections over $(PQ_{\phi})^{-1}\kk[T]$, and taking $\Hom$ commutes
  with base change, we can consider section-wise $\Hom(L_\phi,X)$ as a sheaf of modules over
  $(PQ_{\phi})^{-1}\kk[T]$. Obviously, one can find a polynomial $Q'_{\phi}$ such that
  $\Hom(L_\phi,X)$ is a locally free sheaf over $(PQ_{\phi}Q'_\phi)^{-1}\kk[T]$.
  Therefore, the rank of $\Hom(L_\phi,X)$ is the same for all $t\in\kk\setminus Z(PQ_\phi Q'_\phi)$.
  So it is sufficient to take $Q = \prod_{\phi\in \filt{\dom{P}}{m}} Q_\phi Q'_\phi$ to get the same character.
\end{proof}

Take a construction such as $M(\phi)$, $\phi\in\dom{P}_{\infty}$,
and a grading $\epsilon = k\Lambda_0 + j\delta$. It is easy to see that the $\epsilon$-grading component
$[\epsilon]M(\phi)$ of $M(\phi)$ is a construction that takes sections $L_\phi$, $V$, and $V^*$,
and then makes a finite number of operations like taking sums, tensor products,
kernels and cokernels. Obviously the sum or the tensor product commutes with base change.
This is not true for the kernels and cokernels, but for any application of a kernel or a cokernel
we can localize in a polynomial to retain the section-wise consistency. Therefore,
$[\epsilon]M(\phi)$ can be considered as a valid section of our category sheaf $\drep^c(G,-)$
over some $Q^{-1}\kk[T]$, $Q\in\kk[T]$. Now using the proposition,
we get that the character of $[\epsilon]M(\phi)$ is coherent across all $t\in\kk\setminus Z(Q)$.
The same trick obviously works for $L(\phi)$ and the cohomologies of $\lie{u}^{-}$.
The following theorem is the restatement of the above conclusion.

\begin{theorem}
  \label{stm:m_l_coherence}
  Given an object $S = M(\phi)$, $L(\phi)$, or $H^i(\lie{u}^{-})$
  and a grading $\epsilon = k\Lambda_0+j\delta$,
  the grading component $[\epsilon]S$ has the same character for all
  $t\in \kk\setminus Z(Q)$ and $n\in I\setminus Z(Q)$, where $Q\in\kk[T]$ is
  a nonzero polynomial.
\end{theorem}

\subsection{Limit Dynkin diagrams}
\label{sec:limit_dynkin}

Consider a Dynkin diagram $\aff{X}_{\infty}$ which is equal to $A_{\infty}$ in
the case $GL$, to $D_{\infty}$ in the case $O$, and to $C_{\infty}$ in the case $Sp$.
Denote the objects corresponding to the diagram with a hat and a subscript
$\infty$, for example, denote the root system by $\aff{\Delta}_{\infty}$.
A description of the corresponding objects follows.
Since we are not bound to the corresponding Lie algebras like
$\lie{gl}_{\infty}$, we will use a weight lattice that is more suitable to our needs.
For a usual treatment of these objects, see for example \cite[Section 7.11]{Kac1983}.

Set the weight lattice $\aff{P}_{\infty}$ to be $\Z\{\eps_i\mid i\in B\}\oplus \Z
\Lambda_0\oplus \Z\delta$, where $B$ is equal to $\Z$ in the case $GL$, and is
equal to $\Z_{>0}$ in the cases $O$ and $Sp$. By $E_i$, $i\in B$, we denote the
linear function on $\aff{P}_{\infty}$ that evaluates to the coefficient of
$\eps_i$ in an element. Also by $c$ we denote the linear function on $\aff{P}_{\infty}$
such that $c(\eps_i) = 0$ for all $i\in B$, $c(\delta) = 0$, and the value
$c(\Lambda_0)$ is equal to $1$ in the cases $GL$ and $O$, and to $2$ in 
the case $Sp$.

\begin{center}\bf
  Case $GL$
\end{center}

\bgroup
\def\arraystretch{1.2}
\begin{tabular}{>{$}l<{$}}
  \aff{P}_{\infty} = \Z\{\eps_i \mid i\in \Z\}\oplus \Z\Lambda_0\oplus \Z\delta,\\
  \aff{\Delta}^{+}_{\infty} = \{\eps_i-\eps_j + \delta  \mid i\leq 0<1\leq j\in\Z \}
  \cup \{ \eps_i-\eps_j\mid \text{any other pair $i<j\in\Z$} \},\\
  \aff{\Pi}_{\infty} = 
  \{\alpha_0 := \eps_0 - \eps_{1} + \delta \}
  \cup 
  \{\alpha_i := \eps_i - \eps_{i+1} \mid i\neq 0\in\Z \},\\
  \aff{\Pi}^{\vee}_{\infty} =
  \{ \alpha^{\vee}_0 := E_0 - E_1 + c \}
  \cup
  \{ \alpha_i^{\vee} := E_{i}-E_{i+1} \mid i\neq 0\in\Z \},\\
\end{tabular}
\egroup

\begin{center}\bf
  Case $O$
\end{center}

\bgroup
\def\arraystretch{1.2}
\begin{tabular}{>{$}l<{$}}
  \aff{P}_{\infty} = \Z\{\eps_i \mid 1\leq i\in \Z\}\oplus \Z\Lambda_0\oplus \Z\delta,\\
  \aff{\Delta}^{+}_{\infty} = \{\eps_i-\eps_j, \delta -\eps_i-\eps_j \mid 1\leq i<j\in\Z  \},\\
  \aff{\Pi}_{\infty} =
  \{ \alpha_0 := \delta -\eps_1-\eps_2 \}
  \cup
  \{ \alpha_i :=  \eps_i - \eps_{i+1} \mid 1\leq i\in\Z \},\\
  \aff{\Pi}^{\vee}_{\infty} = \{ \alpha_0^{\vee} := c - E_{1} - E_{2} \}\cup \{ \alpha_i^{\vee} := E_{i}-E_{i+1} \mid 1\leq i\in\Z \}, \\
\end{tabular}
\egroup

\begin{center}\bf
  Case $Sp$
\end{center}

\bgroup
\def\arraystretch{1.2}
\begin{tabular}{>{$}l<{$}}
  \aff{P}_{\infty} = \Z\{\eps_i \mid 1\leq i\in \Z\}\oplus \Z\Lambda_0\oplus \Z\delta,\\
  \aff{\Delta}^{+}_{\infty} = \{\eps_i-\eps_j,  \delta-\eps_i-\eps_j \mid 1\leq i<j\in\Z  \}
  \cup \{ \delta-2\eps_i \mid 1\leq i\in\Z \},\\
  \aff{\Pi}_{\infty} = \{ \alpha_0 := \delta-2\eps_1 \}\cup \{ \alpha_i := \eps_i - \eps_{i+1} \mid 1\leq i\in\Z \},\\
  \aff{\Pi}^{\vee}_{\infty} = \{ \alpha_0^{\vee} := \frac{c}{2} - E_{1} \}\cup \{ \alpha_i^{\vee} := E_{i}-E_{i+1} \mid 1\leq i\in\Z \},\\
\end{tabular}
\egroup

\smallskip

Denote the subdiagram of $\aff{X}_{\infty}$ consisting of the vertex
corresponding to $\alpha_0$
by $Y_{\infty}$, and the complement of $Y_{\infty}$ by $X_{\infty}$.
Also recall that each diagram $\aff{X}_n$ contains the affine
vertex, corresponding to $\alpha_0$, which is the complement of the subdiagram $X_n$.
Denote the
subdiagram consisting of the affine vertex of $\aff{X}_n$ by $Y_n$.

We argue that $\aff{X}_{\infty}$ is the limit of diagrams $\aff{X}_n$ around
$Y$'s. To make sense of this statement, we need to introduce some definitions.
Let $\aff{X}$ be an arbitrary Dynkin diagram.

\begin{definition}
  Given two vertices $x,y\in\aff{X}$, define the
  distance $d(x,y)$ between them as the minimal number of edges required to
  connect $x$ and $y$. 
\end{definition}

\begin{definition}
  For any two subsets $Z,T\subset\aff{X}$, define the distance $d(Z,T)$ between
  them as the Hausdorff distance between $Z$ and $T$.
\end{definition}

\begin{definition}
  Given a subdiagram $Y\subset\aff{X}$ and a
  nonnegative integer $j\in\Z_{\geq 0}$, we define a subdiagram
  $\afilt{Y}{j}\subset\aff{X}$ to be the subdiagram of all vertices of $\aff{X}$
  on the distance less than $j$ from $Y$.
\end{definition}

The first piece of the limit statement is the following lemma.

\begin{lemma}
  \label{stm:neighbor_diagram_same}
  For any $m\in\Z_{\geq 0}$ and $n\in I\setminus Z(\bad_m)$,
  the diagrams $\afilt{Y_n}{m-2}\subset \aff{X}_n$ and $\afilt{Y_\infty}{m-2}\subset \aff{X}_{\infty}$
  are isomorphic.
\end{lemma}

\begin{proof}
  Obvious by inspection.
\end{proof}

The second part is to deal with the weight lattices. For a nonnegative
$m\in\Z_{\geq 0}$, let $\filt{P}{m}$ be $\Z^m\oplus \Z^m$ in the case $GL$, and $\Z^m$
in the cases $O$ and $Sp$. It naturally contains $\filt{\dom{P}}{m}$ with the
injection
$$
[\mu,\nu]\in \filt{\dom{P}}{m} \mapsto (\mu_1,\ldots,\mu_m) \oplus (\nu_1,\ldots,\nu_m)\in \filt{P}{m}
$$
in the case
$GL$, and
$$
\lambda\in \filt{\dom{P}}{m} \mapsto (\lambda_1,\ldots,\lambda_m)\in\filt{P}{m}
$$
in the
cases $O$ and $Sp$. It is easy to see that $\filt{P}{m}$ embeds into $P_n$ for
all $n\in I\setminus Z(\bad_m)$, with
$$
(\alpha_1,\ldots,\alpha_m) \oplus (\beta_1,\ldots,\beta_m)\in\filt{P}{m} \mapsto
\sum_{i=1}^m \alpha_i\eps_i - \sum_{j=1}^m \beta_j \eps_{n+1-j}\in P_n
$$
in the case $GL$,
and
$$
(\beta_1,\ldots,\beta_m)\in\filt{P}{m} \mapsto
\sum_{i=1}^m \beta_i\eps_i\in P_n
$$
in the cases $Sp$ and $O$.
It also embeds into $\aff{P}_{\infty}$, with 
$$
(\alpha_1,\ldots,\alpha_m) \oplus (\beta_1,\ldots,\beta_m)\in\filt{P}{m} \mapsto
\sum_{i=1}^m \alpha_i\eps_i - \sum_{j=1}^m \beta_j \eps_{1-j}\in \aff{P}_{\infty}
$$
in the case $GL$,
and
$$
(\beta_1,\ldots,\beta_m)\in\filt{P}{m} \mapsto
\sum_{i=1}^m \beta_i\eps_i\in \aff{P}_{\infty}
$$
in the cases $Sp$ and $O$.
Finally, define $\filt{\aff{P}}{m}$ as
$\filt{P}{m}\oplus\Z\Lambda_0\oplus\Z\delta$. It embeds into $\aff{P}_n$ for all
$n\in I\setminus Z(\bad_m)$, and into $\aff{P}_{\infty}$.

Consider the embedding of $\filt{\aff{P}}{m}$ into some $\aff{P}_n$. We can see
that the root system of $\afilt{Y_n}{m-2}$ is contained in the image of this
embedding. Moreover, all the coroots obviously act on the image, since it is
just a sublattice. Therefore, the embedding provides $\filt{\aff{P}}{m}$ with a
structure of a weight lattice of $\afilt{Y_n}{m-2}$. The same trick can be done
with $\afilt{Y_{\infty}}{m-2}$. But all the diagrams $\afilt{Y_n}{m-2}$ and
$\afilt{Y_\infty}{m-2}$ are actually isomorphic due to
\cref{stm:neighbor_diagram_same}, so we can compare this weight lattice structures.

\begin{lemma}
  The induced weight lattice structure on $\filt{\aff{P}}{m}$ is the same for
  all $n$, and $\infty$.
\end{lemma}

\begin{proof}
  Can be seen directly from the provided explicit descriptions of the roots and
  coroots in the finite and infinite cases.
\end{proof}

This two lemmas together justify why $\aff{X}_{\infty}$ can be considered as the
limit of $\aff{X}_n$ around $Y$'s. 

\subsection{Stable alternating sum}
\label{sec:additional_weyl_facts}

Consider an arbitrary Dynkin diagram $\aff{X}$.  
Recall that the support $\supp \alpha$ of a root $\alpha = \sum_{x\in \aff{X}} k_x\alpha_x\in
\aff{\Delta}$ is the subset of $\aff{X}$ consisting of all vertices $x$ with
$k_x\neq 0$. 

\begin{lemma}
  \label{stm:support_crawls}
  Given a root $\alpha\in\aff{\Delta}$ and a Weyl group element $w\in\aff{W}$,
  we have
  \begin{equation*}
  d(\supp\alpha, \supp w\alpha) \leq l(w).
  \end{equation*}
\end{lemma}

\begin{proof}
  By the triangle inequality, it is sufficient to prove the claim in the case $w = s_x$, $x\in\aff{X}$.
  If $\supp\alpha = \supp s_x\alpha$, there is nothing to prove.
  Otherwise $\supp\alpha$ and $\supp s_x\alpha$ can only differ by $x$. 
  We can assume that $\supp s_x\alpha = \supp\alpha\cup \{ x \}$, because
  $\alpha$ and $s_x\alpha$ are interchangable.
  Since $s_x\alpha \neq \alpha$, it follows that $\langle \alpha,
  \alpha^{\vee}_x \rangle \neq 0$, so $x$ is adjacent to one of the vertices in
  the support of $\alpha$. Therefore $d(\supp\alpha, \supp s_x\alpha) = 1$.
\end{proof}

Consider an inclusion $X\subset \aff{X}$ as in \cref{sec:bourbaki_stuff}.
Denote the subdiagram $\aff{X}\setminus X\subset \aff{X}$ by $Y$.

\begin{lemma}
  \label{stm:zero_wx_is_stable}
  Take an element $w\in\wx$ with the length $j$. Then $w\in W(\afilt{Y}{j})$.
\end{lemma}

\begin{proof}
  Take any reduced expression $w = s_{x_1} \ldots s_{x_j}$,
  $x_i\in \aff{X}$. Fix a particular $1\leq i\leq j$.
  By \cref{stm:full_description_of_negatives} the root 
  $s_{x_1}\ldots s_{x_{i-1}}(\alpha_{x_i})$ is an inversion root of $w$.
  From the definition of $\wx$ we know that the intersection of $I(w)$ and $\Delta(X)$ is empty.
  Therefore there is a vertex $y\in Y$ such that $y \in \supp s_{x_1}\ldots s_{x_{i-1}}(\alpha_{x_i})$.
  Applying \cref{stm:support_crawls} we get $d(\{ x_i \}, \supp s_{x_1}\ldots
  s_{x_{i-1}}(\alpha_{x_i})) \leq i-1$. In particular, $d(x_i, y) \leq i-1 < j$.
  Thus $x_i \in \afilt{Y}{j}$.
\end{proof}

For an element $\gamma = \sum_{x\in \aff{X}} c_x\alpha_x \in \aff{Q}$ in the
root lattice, define the $Y$-coefficient of 
$\gamma$ to be the sum $\sum_{y\in Y} c_y$.

\begin{lemma}
  \label{stm:any_wx_is_stable}
  Take a dominant weight $\phi\in\aff{P}^{+}$ and an element $w\in\wx$.
  Suppose that the $Y$-coefficient of $\phi - w\cdot\phi$ is $j$.
  Then $w\in W(\afilt{Y}{j})$.
\end{lemma}

\begin{proof}
  We have $\phi - w\cdot\phi = \phi - w\phi - w\cdot 0$.
  By \cref{stm:weyl_take_dominant_lower}, the $Y$-coefficient of $\phi - w\phi$
  is nonnegative. Therefore, the $Y$-coefficient of $-w\cdot 0$ is at
  most $j$. By \cref{stm:dot_action_by_negatives_sum}, $-w\cdot 0 =
  \sum_{\alpha\in I(w)} \alpha$. Also, since $w$ lies in
  $\wx$, the $Y$-coefficient of every summand is
  at least $1$. Therefore, $I(w)$ contains at most $j$ roots.
  From \cref{stm:full_description_of_negatives} we get that $l(w) \leq j$.
  Finally, the claim follows by \cref{stm:zero_wx_is_stable}.
\end{proof}

We will now apply this results to prove the coherence of the alternating sums.

\begin{proposition}
  \label{stm:series_are_coherent}
  Suppose we have a dominant weight $\phi \in \dom{\aff{P}}_{\infty}$.
  Then the family of characters
  \begin{gather*}
   \left\{ \sum_{w\in \wx_{\infty}} (-1)^{l(w)} \chr
   L_{w\cdot\phi}\in K(\drep(G_t)^a)\mid t\in\kk\setminus\Z\right\} \cup \\
   \cup \left\{\sum_{w\in
    \wx_n} (-1)^{l(w)} \chr L_{w\cdot\phi}\in K(\rep(G_n^0)^a)\mid n\in I\right\}
  \end{gather*}
  is coherent.
\end{proposition}

\begin{proof}
  Fix a grading $\epsilon$.
  Note that with the diagrams $\aff{X}_n$ and $\aff{X}_{\infty}$,
  $Y$-coefficient is just the coefficient of $\delta$.
  Therefore, we can apply \cref{stm:any_wx_is_stable} to get that the coefficients
  \begin{gather*}
    [\epsilon]\left(\sum_{w\in \wx_{\infty}} (-1)^{l(w)} \chr L_{w\cdot\phi}\right),\\
    [\epsilon]\left(\sum_{w\in \wx_{n}} (-1)^{l(w)} \chr L_{w\cdot\phi}\right),\\
  \end{gather*}
  actually both coincide with the coefficient
  $$
    [\epsilon]\left(\sum_{w\in W(\afilt{Y}{j})^X} (-1)^{l(w)} \chr L_{w\cdot\phi}\right)
  $$
  for some $j$.  
  Since this sum is finite, we characters are the same over $Q^{-1}\kk[T]$
  for some nonzero polynomial $Q\in\kk[T]$
\end{proof}

\subsection{Stable formulas}
\label{sec:stable_formulas}

Here we use the previous subsection results to prove the stable counterparts of
the classical formulas.

\begin{theorem}
  \label{stm:limit_transcendental}
  Suppose we have a dominant weight $\phi \in \dom{\aff{P}}_{\infty}$.
  Then in $K(\drep(G_t)^a)$ we have
  \begin{equation*}
    \chr L(\phi) = \sum_{w\in \wx_{\infty}} (-1)^{l(w)} \chr M(w\cdot\phi)
  \end{equation*}
  for any $t\in \kk$ transcendental over $\Q$.
\end{theorem}

\begin{proof}
  From \cref{stm:m_l_coherence} we know that the LHS is a coherent family,
  and from \cref{stm:m_l_coherence} and \cref{stm:series_are_coherent} we know
  that the RHS is a coherent family.

  Consider a grading $\epsilon$.
  Take the element $c\in K_m$ representing the character
  of the $\epsilon$-graded component of the LHS
  for all $t\in \kk\setminus Z(Q)$
  and $n\in I\setminus Z(Q)$,
  and the element $c'\in K_m$ representing the character
  of the $\epsilon$-graded component of the RHS
  for all $t\in \kk\setminus Z(Q')$
  and $n\in I\setminus Z(Q')$.

  Since $t$ is transcendental over $\Q$, we
  have $Q(t)Q'(t)\neq 0$. Take a nonnegative integer $n\in I$ such that
  $Q(n)Q'(n)\neq 0$.
  From \cref{stm:relative_finite_kac} we deduce that $c = c'$.
  Therefore $[\epsilon]\str{LHS} =
  [\epsilon]\str{RHS}$. The claim follows.
\end{proof}

\begin{corollary}
  In $K(\drep(G_t)^a)$ we have
  $$
  \frac{1}{\chr M(0)} = \sum_{w\in \wx_{\infty}} (-1)^{l(w)} \chr L_{w\cdot 0}
  $$
  for any $t\in \kk$ transcendental over $\Q$.
\end{corollary}

\begin{proof}
  Put $\phi=0$ in \cref{stm:limit_transcendental}, and use that $L(0)$ is the
  trivial module $\triv$.
\end{proof}

We also have the stable variant of the Garland formula.

\begin{theorem}
  For a nonnegative $i\in\Z_{\geq 0}$, the $i$-th cohomology
  $H^i(\lie{u}^{-}_{t})$ of $\lie{u}^{-}_{t}$ is isomorphic to
  $$
  H^i(\lie{u}^{-}) \iso \bigoplus_{\substack{w\in \wx_{\infty} \\ l(w) = i}}
  L_{w\cdot 0}
  $$
  as an $\lie{l}_t$-module
  for any $t\in \kk\setminus Z(p)$ outside of the zeros of some polynomial
  $p\in\Q[T]$. 
\end{theorem}

\begin{proof}
  The proof just repeats the argument from \cref{stm:limit_transcendental}.
\end{proof}

\section{Explicit calculations}

\subsection{Expressing reduced Weyl group elements using root system}
\label{sec:explicit_reduced}

Consider an arbitrary Dynkin diagram $\aff{X}$.

\begin{definition}
  Say that a \emph{finite} set $S\subset \aff{\Delta}^{+}$ of positive roots is a
  \emph{slice} if $S\cup \aff{\Delta}^{-}\setminus (-S)$ is closed under addition.
\end{definition}

The following lemma is a direct rewrite of the similar Bourbaki's theorem
applied to the case when $\aff{X}$ is not necessarily finite. 

\begin{lemma}
  Suppose we have a subset of positive roots $S\subset\aff{\Delta}^{+}$.
  Then $S$ is a slice if and only if there exists $w\in\aff{W}$ such that
  $S = I(w)$.
\end{lemma}

\begin{proof}
  If such element $w\in\aff{W}$ exists, then $I(w)$ 
  is finite by \cref{stm:full_description_of_negatives}, and $I(w)\cup
  \aff{\Delta}^{-}\setminus (-I(w)) = w(\aff{\Delta}^{-})$ is closed under addition.
  
  Suppose we have a slice $S$.
  Denote the union $S\cup \aff{\Delta}^{-}\setminus
  (-S)$ by $C$.
  We will find $w$ such that $I(w) = S$ using the following algorithm. We begin
  with $w = e$. On each step, we will refine current $w$ to be closer to the
  one we need, where the closeness is measured by the cardinality of the
  symmetric difference $|C \mathop{\triangle} w(\aff{\Delta}^{-})|$.

  Suppose that for our current $w$, there exists a simple root
  $\alpha$ such that $w(-\alpha)\notin C$ for some simple root $\alpha$.
  Note that $C\cup -C = \aff{\Delta}$, therefore we have 
  $w(\alpha)\in C$. Conjugating
  \cref{stm:simple_reflection_permutes_almost} by $w$, we get that the
  reflection $s_{w(\alpha)}$ in the root $w(\alpha)$ permutes all the
  elements of $w(\aff{\Delta}^{-})$ except for $-w(\alpha)$, which it
  sends to $w(\alpha)$. Put $w' = s_{w(\alpha_i)} w$. We get that
  $w'(\aff{\Delta}^{-}) = (w(\aff{\Delta}^{-})\setminus \{ -w(\alpha_i) \})\cup
  \{ w(\alpha_i) \}$.
  Therefore the symmetric difference of $C$ and $w'(\aff{\Delta}^{-})$ has less
  elements. Repeat the algorithm with $w'$.

  Since the cardinality of the symmetric difference is less at each step,
  at some point we fall into the second case, meaning that $w(-\alpha)\in C$ for every simple root
  $\alpha$. Now since $C$ is closed under addition, $w(\beta)\in C$ for any
  negative root $\beta\in\aff{\Delta}^{-}$. It means that
  $w(\aff{\Delta}^{-})\subset C$. Therefore $C=w(\aff{\Delta}^{-})$, and $S = I(w)$.
\end{proof}

Now consider the diagram $\aff{X}_{\infty}$.
Our goal is to describe the slices corresponding to the elements $\wx_{\infty}$.

\begin{lemma}
  \label{stm:reduced_using_roots}
  A finite subset $S\subset\aff{\Delta}^{+}_{\infty}$ of positive roots is a
  slice corresponding to an element in $\wx_{\infty}$
  if and only if the following conditions hold:
  \begin{enumerate}
  \item $S\cap\Delta^{+}_{\infty} = \emptyset$,
  \item for any element $\alpha\in S$ and positive root
    $\beta\in\Delta^{+}_{\infty}$ such that $\alpha-\beta$ is a root, we have
    $\alpha-\beta \in S$.
  \end{enumerate}
\end{lemma}

\begin{proof}
Recall from the explicit description of the objects of $\aff{X}_{\infty}$ that
the elements of the root system $\aff{\Delta}_{\infty}$ can only have three
different $\delta$-coefficients. The root subsystem $\Delta_{\infty}$ are
exactly the roots with zero $\delta$-coefficient. The positive roots
$\aff{\Delta}^{+}_{\infty}\setminus \Delta^{+}_{\infty}$ form the set of roots with
$\delta$-coefficient $1$. Finally, the opposites
$\aff{\Delta}^{-}_{\infty}\setminus \Delta^{-}_{\infty}$
form the set of roots with $\delta$-coefficient $-1$.

  First, let us prove that any slice $S$ of an $X$-reduced Weyl group element satisfy these conditions.
  The first condition holds since $S$ corresponds to an $X$-reduced Weyl group element.
  Consider the second condition. Say we have
  $\alpha\in S$ and $\beta\in\Delta^{+}_{\infty}$ such that $\alpha-\beta$ is a
  root. $\beta$ is a root with $\delta$-coefficient $0$, therefore
  $-\beta$ lies in $\aff{\Delta}^{-}_{\infty}\setminus (-S)$. Since $S$ is a slice,
  the set $S\cup \aff{\Delta}^{-}_{\infty}\setminus (-S)$ is closed under addition,
  therefore $\alpha-\beta \in S\cup \aff{\Delta}^{-}_{\infty}\setminus (-S)$.
  But the $\delta$-coeffitient of $\alpha-\beta$ is $1$, therefore
  $\alpha-\beta\notin \aff{\Delta}^{-}_{\infty}$, and $\alpha-\beta\in S$.

  Now let $S\subset \aff{\Delta}^{+}_{\infty}$ be a finite subset of positive roots
  satisfying the two conditions. We need to prove that the set $C = S\cup
  \aff{\Delta}^{-}_{\infty}\setminus (-S)$ is closer under addition.
  Suppose we have roots 
  $\alpha,\beta\in C$ such that $\alpha+\beta$ is a root. Taking into
  account that the $\delta$-coefficient of each root $\alpha,\beta,\alpha+\beta$
  lies in the set $\{-1,0,1\}$, we prove that $\alpha+\beta\in C$ by
  considering several cases. We can also assume that
  the $\delta$-coefficient of $\alpha$ is greater or equal than the one of $\beta$.

  $\bullet$\ The $\delta$-coefficients of $\alpha,\beta,\alpha+\beta$ are $0,0,0$ respectively.
  Since the only roots in $C$ with $\delta$-coeffitient $0$ are $\Delta^{-}_{\infty}$,
  we have $\alpha,\beta\in\Delta^{-}_{\infty}$. Then $\alpha+\beta\in\Delta^{-}_{\infty}\subset C$.

  $\bullet$\ The $\delta$-coefficients of $\alpha,\beta,\alpha+\beta$ are $1,-1,0$ respectively.
  Since $\alpha+\beta$ is a root, it is either a positive or a negative one.
  If it is negative, we are done. Suppose it is positive. Note that $\alpha\in
  S$. Then $\alpha-(\alpha+\beta) = -\beta\in C$ because of the condition $(2)$.
  Contradiction with the fact that $\beta\in C$.

  $\bullet$\ The $\delta$-coefficients of $\alpha,\beta,\alpha+\beta$ are $1,0,1$ respectively.
  Note that $\alpha\in S$ and $\beta\in\Delta^{-}_{\infty}$. Then
  $\alpha+\beta = \alpha - (-\beta)\in S\subset C$ by the condition $(2)$ of $S$.

  $\bullet$\ The $\delta$-coefficients of $\alpha,\beta,\alpha+\beta$ are $0,-1,-1$ respectively.
  Suppose $\alpha+\beta\notin C$. Then $\alpha+\beta\in -S$, so
  $-\alpha-\beta\in S$. Also note that $\alpha\in\Delta^{-}_{\infty}$. Then we
  have $(-\alpha-\beta)-(-\alpha) = -\beta \in S\subset C$ by the condition
  $(2)$. Contradiction with the fact that $\beta\in C$.

  Therefore $S$ is a slice. By the condition $(1)$ it corresponds to an
  $X$-reduced Weyl group element.
\end{proof}

This motivates the following definition.

\begin{definition}
  For two roots $\alpha,\beta\in\aff{\Delta}^{+}_{\infty}\setminus
  \Delta^{+}_{\infty}$ with $\delta$-coefficient $1$, say that $\alpha\leq
  \beta$ if there exists a sequence
  $\gamma_1,\ldots,\gamma_k\in\Delta^{+}_{\infty}$ of positive roots with
  $\delta$-coefficent $0$ such that $\alpha = \beta - \gamma_1 - \ldots - \gamma_k$.
\end{definition}

\begin{corollary}
  The set $\wx_{\infty}$ of $X$-reduced Weyl group elements is in a bijection
  with the set of finite downward closed subsets of $\aff{\Delta}^{+}_{\infty}\setminus
  \Delta^{+}_{\infty}$ via
  $$
  w\in\wx_{\infty} \mapsto I(w)\subset \aff{\Delta}^{+}_{\infty}\setminus
  \Delta^{+}_{\infty}.
  $$
\end{corollary}

Let us use this correspondence to describe $X$-reduced slices.
We will produce the explicit formulas only in the $GL$ case.

\begin{center}\bf
  Case $GL$
\end{center}

\begin{lemma}
  The poset $\aff{\Delta}^{+}_{\infty}\setminus \Delta^{+}_{\infty}$ is
  isomorphic to the poset $\Z_{>0}^2$ via the map 
  $$
  (i,j)\in\Z_{>0}^2 \mapsto \delta
  + \eps_{1-j} - \eps_i\in \aff{\Delta}^{+}_{\infty}\setminus \Delta^{+}_{\infty},
  $$
  where in $\Z_{>0}^2$ we have $(i,j)\leq (i',j')$ if and only
  if $i\leq i'$ and $j\leq j'$.
\end{lemma}

\begin{proof}
  It is clear by examination of the explicit description of the root system.
\end{proof}

\begin{corollary}
  For a partition $\lambda$, let $S_{\lambda}\subset \aff{\Delta}^{+}_{\infty}\setminus
  \Delta^{+}_{\infty}$ be the subset that contains a root $\delta +
  \eps_{1-j} - \eps_i$ if and only if $\lambda_i \geq j$.
  The downward closed subsets of $\aff{\Delta}^{+}_{\infty}\setminus
  \Delta^{+}_{\infty}$ are exactly the slices $S_{\lambda}$, where $\lambda$ is
  any partition.
\end{corollary}

\begin{lemma}
  Fix a partition $\lambda$. Let $\lambda =
  (p_1,\ldots,p_b \mid q_1,\ldots,q_b)$ be the Frobenius coordinates of $\lambda$.
  Let $w_{\lambda}$ be the $X$-reduced Weyl group
  element such that $I(w_{\lambda}) = S_{\lambda}$.
  Denote the increasing sequence of numbers in $\Z_{\geq 0} \setminus \{ p_i \}_{i=1}^b$ by $\{ \bar{p}_c
  \}_{c=1}^{\infty}$,
  and the respective complement of $q$'s by $\{ \bar{q}_c \}_{c=1}^{\infty}$.
  Also fix a weight $\beta = \sum_{i=-\infty}^{+\infty} \beta_i \eps_i + k\Lambda_0 + j\delta$.
  The following is true:
  \begin{enumerate}
    \item the action of $w_\lambda$ on $\beta$ is described by
      \begin{gather*}
      w_\lambda \beta = \sum_{c=1}^{\infty} \beta_{-\bar{p}_c}\eps_{-c-b+1} +\sum_{i=1}^b
      (\beta_{q_i+1}-k)\eps_{1-i} + \\
      + \sum_{i=1}^b (k+\beta_{-p_i})\eps_i + \sum_{c=1}^{\infty}
      \beta_{\bar{q}_{c}+1}\eps_{b+c} + \\
        + k\Lambda_0 + \left(j - kb + \sum_{i=1}^b (- \beta_{-p_i} + \beta_{q_i+1})\right)\delta ,
      \end{gather*}
    \item the dot action of $w_\lambda$ on $\beta$ is described by
      \begin{gather*}
      w_\lambda \cdot\beta = \sum_{c=1}^{\infty} (\beta_{-\bar{p}_c}+\bar{p}_c-c-b+1)\eps_{-c-b+1} +\sum_{i=1}^b
      (\beta_{q_i+1}-k-i-q_i)\eps_{1-i} + \\
      + \sum_{i=1}^b (k+\beta_{-p_i}+i+p_i)\eps_i + \sum_{c=1}^{\infty}
      (\beta_{\bar{q}_c+1} - \bar{q}_c + b + c-1)\eps_{b+c} + \\
        + k\Lambda_0 + \left(j - (k+1)b + \sum_{i=1}^b (-\beta_{-p_i} +
          \beta_{q_i+1} - p_i - q_i) \right)\delta ,
      \end{gather*}
    \item the weight $w_\lambda\cdot 0$ can be also described using $\lambda$ as
      follows:
      \begin{gather*}
        w_\lambda \cdot 0 = \sum_{i=1}^{\infty} -\lambda^t_i\eps_{1-i} +
        \sum_{i=1}^{\infty} \lambda_i\eps_i - |\lambda| \delta,
      \end{gather*}
    \item the element $w_\lambda$ has a reduced decomposition
      \begin{gather*}
      w_\lambda =\\
      s_0 s_{-1} \ldots s_{-p_1} s_1 \ldots s_{q_1} \\
      s_0 s_{-1} \ldots s_{-p_2} s_1 \ldots s_{q_2}\\
      \ldots\\
      s_0 s_{-1} \ldots s_{-p_b} s_1 \ldots s_{q_b}.\\
      \end{gather*}
  \end{enumerate}
\end{lemma}

\begin{proof}
  Easy to see using induction on $|\lambda|$.
\end{proof}

\begin{center}\bf
  Case $O$
\end{center}

\begin{lemma}
  The poset $\aff{\Delta}^{+}_{\infty}\setminus \Delta^{+}_{\infty}$ is
  isomorphic to the poset $\{(i,j) \mid 1\leq  i < j\}$ via the map 
  $$
  (i,j) \mapsto \delta
  - \eps_{i} - \eps_j\in \aff{\Delta}^{+}_{\infty}\setminus \Delta^{+}_{\infty},
  $$
  where we have $(i,j)\leq (i',j')$ if and only
  if $i\leq i'$ and $j\leq j'$.
\end{lemma}

\begin{proof}
  It is clear by examination of the explicit description of the root system.
\end{proof}

\begin{corollary}
  For a decreasing sequence $\{ p_i \}_{i=1}^b$, let $S_{p}\subset \aff{\Delta}^{+}_{\infty}\setminus
  \Delta^{+}_{\infty}$ be the subset that contains a root $\delta -
  \eps_{i} - \eps_j$ if and only if $p_i + i \geq j$.
  The downward closed subsets of $\aff{\Delta}^{+}_{\infty}\setminus
  \Delta^{+}_{\infty}$ are exactly the slices $S_{p}$.
\end{corollary}

\begin{center}\bf
  Case $Sp$
\end{center}

\begin{lemma}
  The poset $\aff{\Delta}^{+}_{\infty}\setminus \Delta^{+}_{\infty}$ is
  isomorphic to the poset $\{(i,j) \mid 1\leq i \leq j\}$ via the map 
  $$
  (i,j) \mapsto \delta
  - \eps_{i} - \eps_j\in \aff{\Delta}^{+}_{\infty}\setminus \Delta^{+}_{\infty},
  $$
  where we have $(i,j)\leq (i',j')$ if and only
  if $i\leq i'$ and $j\leq j'$.
\end{lemma}

\begin{proof}
  It is clear by examination of the explicit description of the root system.
\end{proof}

\begin{corollary}
  For a decreasing sequence $\{ q_i \}_{i=1}^b$, let $S_{q}\subset \aff{\Delta}^{+}_{\infty}\setminus
  \Delta^{+}_{\infty}$ be the subset that contains a root $\delta +
  \eps_{i} - \eps_j$ if and only if $q_j + j \geq i$.
  The downward closed subsets of $\aff{\Delta}^{+}_{\infty}\setminus
  \Delta^{+}_{\infty}$ are exactly the slices $S_{q}$.
\end{corollary}

\subsection{Explicit formulas}
\label{sec:explicit_formulas}

We will only explore the combination of the results from
\cref{sec:stable_formulas,sec:explicit_reduced} in the case $GL$.

Putting the explicit description of $\wx_{\infty}$ into
\cref{stm:limit_transcendental}, we get the following.
Let $q$ denote the character of $z^{-1}$.

\begin{theorem}
  \label{stm:limit_transcendental_exp}
  Suppose we have two partitions $\mu$ and $\nu$,
  and a nonnegative integer $k$ such that $k\geq \mu_1+\nu_1$.
  Then in $K(\drep(G_t)^a)$ we have
  \begin{equation*}
    \chr L([\mu,\nu]+k\Lambda_0) = \sum_{\text{partition}\ \lambda} (-1)^{|\lambda|} q^{\delta(\lambda,\mu,\nu)} \chr M(\lambda\cdot [\mu,\nu]+k\Lambda_0)
  \end{equation*}
  for any $t\in \kk$ transcendental over $\Q$,
  where the given a partition $\lambda$ with the Frobenius coordinates
  $\lambda = (p_1,\ldots,p_b \mid q_1,\ldots,q_b)$, we define
  \begin{gather*}
    \delta(\lambda,\mu,\nu) = |\lambda| + \sum_{i=1}^b (k - \mu_{q_i+1} -
    \nu_{p_i+1}),\\
    \lambda\cdot[\mu,\nu] = \\
    [
    (k-\nu_{p_1+1},k-\nu_{p_2+1},\ldots,k-\nu_{p_b+1},
    \mu_1,\ldots,\hat{\mu_{q_1+1}},\ldots,\hat{\mu_{q_2+1}},\ldots,)
    + \lambda
    ,\\
    (k-\mu_{q_1+1},k-\mu_{q_2+1},\ldots,k-\mu_{q_b+1},
    \nu_1,\ldots,\hat{\nu_{q_1+1}},\ldots,\hat{\nu_{q_2+1}},\ldots,)
    + \lambda^t
    ].
  \end{gather*}
\end{theorem}

In \cite[Section 6]{EtingofII} Etingof has deduced a formula for the character with the
dominant weight $\Lambda_0$, but using different techniques. We wrote a
Python script to check that the two formulas coincide up to $q^{10}$. Here we provide
the series only up to $q^5$, denoting the charactrer of $L_{[\lambda,\mu]}$ by the Young diagrams
of $[\lambda,\mu]$:
\begin{gather*}
    \chr L(\Lambda_0) = \\
    \left[\,\cdot\,,\cdot\,\right] +q
    \left[\,\ydiagram{1}\,,\ydiagram{1}\,\right] +q^2
    \left(\left[\,\cdot\,,\cdot\,\right] + 2
      \left[\,\ydiagram{1}\,,\ydiagram{1}\,\right] +
      \left[\,\ydiagram{1,1}\,,\ydiagram{1,1}\,\right]\right)
    \\
    +q^3 \left(2 \left[\,\cdot\,,\cdot\,\right] + 4
      \left[\,\ydiagram{1}\,,\ydiagram{1}\,\right] + 2
      \left[\,\ydiagram{1,1}\,,\ydiagram{1,1}\,\right] +
      \left[\,\ydiagram{1,1}\,,\ydiagram{2}\,\right] +
      \left[\,\ydiagram{1,1,1}\,,\ydiagram{1,1,1}\,\right] +
      \left[\,\ydiagram{2}\,,\ydiagram{1,1}\,\right]\right)
    \\
    +q^4
    \left(4 \left[\,\cdot\,,\cdot\,\right] + 8 \left[\,\ydiagram{1}\,,\ydiagram{1}\,\right] + 6 \left[\,\ydiagram{1,1}\,,\ydiagram{1,1}\,\right] + 2 \left[\,\ydiagram{1,1}\,,\ydiagram{2}\,\right] + 2 \left[\,\ydiagram{1,1,1}\,,\ydiagram{1,1,1}\,\right] +\right.\\
    \left. + \left[\,\ydiagram{1,1,1}\,,\ydiagram{2,1}\,\right] +
      \left[\,\ydiagram{1,1,1,1}\,,\ydiagram{1,1,1,1}\,\right] + 2
      \left[\,\ydiagram{2}\,,\ydiagram{1,1}\,\right] +
      \left[\,\ydiagram{2}\,,\ydiagram{2}\,\right] +
      \left[\,\ydiagram{2,1}\,,\ydiagram{1,1,1}\,\right]\right)
    \\
    +q^5
    \left(6 \left[\,\cdot\,,\cdot\,\right]+ 16 \left[\,\ydiagram{1}\,,\ydiagram{1}\,\right] + 12 \left[\,\ydiagram{1,1}\,,\ydiagram{1,1}\,\right] + 6 \left[\,\ydiagram{1,1}\,,\ydiagram{2}\,\right] + 6 \left[\,\ydiagram{1,1,1}\,,\ydiagram{1,1,1}\,\right] + 3 \left[\,\ydiagram{1,1,1}\,,\ydiagram{2,1}\,\right] + 2 \left[\,\ydiagram{1,1,1,1}\,,\ydiagram{1,1,1,1}\,\right] + \right.\\
    \left.+ \left[\,\ydiagram{1,1,1,1}\,,\ydiagram{2,1,1}\,\right] +
      \left[\,\ydiagram{1,1,1,1,1}\,,\ydiagram{1,1,1,1,1}\,\right] + 6
      \left[\,\ydiagram{2}\,,\ydiagram{1,1}\,\right] + 2
      \left[\,\ydiagram{2}\,,\ydiagram{2}\,\right] + 3
      \left[\,\ydiagram{2,1}\,,\ydiagram{1,1,1}\,\right] +
      \left[\,\ydiagram{2,1}\,,\ydiagram{2,1}\,\right] +
      \left[\,\ydiagram{2,1,1}\,,\ydiagram{1,1,1,1}\,\right]\right) + \\
    + \ldots
\end{gather*}

Another particular case is the Weyl denominator formula.

\begin{corollary}
  In $K(\drep(G_t)^a)$ we have
  \begin{gather*}
  \frac{1}{\chr M(0)} = \sum_{\text{partition}\ \lambda} (-1)^{|\lambda|}
  q^{|\lambda|}\chr L_{[\lambda,\lambda^t]}
  \end{gather*}
  for any $t\in \kk$ transcendental over $\Q$.
\end{corollary}

Also we can specify the Garland formula.

\begin{theorem}
  For a nonnegative $i\in\Z_{\geq 0}$, the $i$-th cohomology
  $H^i(z^{-1}\lie{sl}_{t}[z^{-1}])$ is isomorphic to
  $$
  H^i(z^{-1}\lie{sl}_{t}[z^{-1}]) \iso \bigoplus_{\substack{\text{partition}\ \lambda \\ |\lambda| = i}}
  L_{[\lambda,\lambda^t]}
  $$
  as an $\lie{sl}_t$-module
  for any $t\in \kk\setminus Z(p)$ outside of the zeros of some polynomial
  $p\in\Q[T]$. 
\end{theorem}

We also want to put the categorical dimensions into the explicit formulas.
In \cite[Section 2.5]{EtingofII}, Etingof provides a formula for
$\dim L_{[\lambda,\mu]}$, but it lacks a manifest independence of $l(\lambda)$ and
$l(\mu)$. We prove another formula for $\dim L_{[\lambda,\mu]}$.

\begin{lemma}
  Given two partitions $\lambda$ and $\mu$, the categorical dimension of
  $L_{[\lambda,\mu]}$ is given by the formula
  $$
  \dim L_{[\lambda,\mu]} =
  \frac{ \prod_{k\in \Z} (t-k)^{m_{\lambda,\mu,k}} }%
  {\prod\limits_{\text{hook $h$ in $\lambda$}} l(h) \prod\limits_{\text{hook $h$ in $\mu$}} l(h)} 
  $$
  for all $t\in \kk\setminus Z(\bad_{|\lambda|+|\mu|})$,
  where
  \begin{gather*}
    m_{\lambda,\mu,k} = \#\{ (i,j) \mid i,j\geq 1, (i-\lambda_i) + (j-\mu_j) = k + 1 \} - k.
  \end{gather*}
\end{lemma}

\begin{proof}
  This is just an interpolation of the usual Weyl dimension formula.
\end{proof}

Now we can use it to get a combinatorial identity from the Weyl denominator formula.
What we get is the Nekrasov-Okounkov hook length formula, so in a sense, we provide a categorical interpretation
of the formula in Deligne's category. For the Nekrasov-Okounkov formula references, see \cite[Formula (6.12)]{NO06},
\cite{Han08}. The analogs of this formula for other types also appear in \cite{Westbury2006}
and \cite{Petr15}.

\begin{theorem}
  We have the identity
  \begin{gather*}
    \prod_{i=1}^{\infty} (1-q^i)^{x-1} = \sum_{\text{partition}\ \lambda}
    q^{|\lambda|} \prod_{\text{hook $h$ in $\lambda$}} \left( 1 - \frac{x}{l(h)^2} \right) .
  \end{gather*}
\end{theorem}

\printbibliography

\end{document}